\documentclass[12pt, psamsfonts]{amsart}

\usepackage[pdftex]{graphicx,epsfig,pict2e}
\usepackage{amsmath,amsfonts,amssymb}
\usepackage{enumerate}

\usepackage{tikz}

\theoremstyle{plain}
\newtheorem{theorem}{Theorem}
\newtheorem{lemma}{Lemma}

\theoremstyle{definition}

\newtheorem{remark}{Remark}

\def\C{\mathcal C}

\title{Spherical Thrackles}
\author{Grant Cairns}
\author{Timothy J.~Koussas}
\author{Yuri Nikolayevsky}

\address{Dept of Mathematics and Statistics, La Trobe University, Melbourne, Australia 3086}
\email{G.Cairns@latrobe.edu.au}
\email{tkoussas@latrobe.edu.au}
\email{Y.Nikolayevsky@latrobe.edu.au}

\thanks{The second author was partially supported by the AMSI Vacation Research Scholarship}

\begin{document}

\maketitle

\begin{abstract} We establish Conway's thrackle conjecture in the case of spherical thrackles; that is, for drawings on the unit sphere where the edges  are  arcs of great circles.
\end{abstract}

\section{Introduction}

Let $G$ be a finite abstract graph. A \emph{thrackle} drawing  is a graph drawing of $G$ on some surface where every pair of distinct edges in $G$ intersects in a single point, either at a common endpoint or at a proper crossing; see \cite{CN,CN2,CN3,FP,LPS,PRT,PS,PP,W}. A \emph{spherical} thrackle drawing is a thrackle drawing  on the unit sphere where the edges  are represented by arcs of great circles.

The class of spherical thrackle drawings is a natural spherical analog of straight-line thrackles drawn on the plane. Despite the similarity, the graphs that can be drawn as spherical thrackles form a larger class than those that can be drawn as straight-line thrackles. Clearly, every graph that can be drawn as a straight-line thrackle can also be drawn as a (sufficiently small) spherical thrackle, but the converse is not true. By the results of Woodall \cite{W}, the only cycles which can be drawn as straight-line thrackles are the odd cycles. In comparison, all even cycles other than the $4$-cycle can be drawn as spherical thrackles; that is, every cycle that has a thrackle drawing in the plane also has a spherical thrackle drawing. Using an adaptation of Woodall's edge-insertion procedure for spherical thrackles \cite{W}, we can obtain from the $6$-cycle drawing the rest of the even cycle drawings, as demonstrated in Figure \ref{F1}.
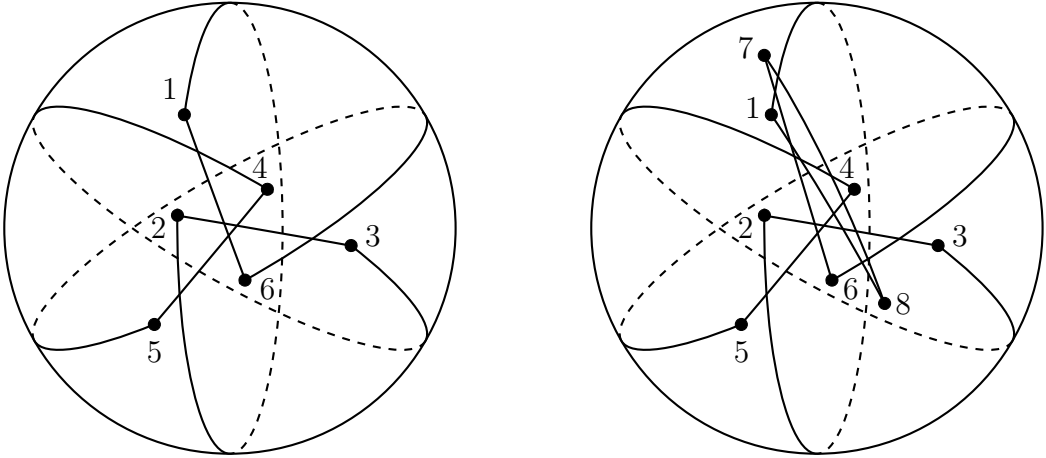
\begin{figure}[h!]
\begin{tikzpicture}
\draw [thick] (0,0) circle (3);
\draw [thick,style=dashed,rotate around={30:(0,0)}] (180:3 and 0.7) arc (180:0:3 and 0.7);
\draw [thick,rotate around={30:(0,0)}] (180:3 and 0.7) arc (180:240:3 and 0.7);
\draw [thick,rotate around={30:(0,0)}] (360:3 and 0.7) arc (360:265:3 and 0.7);
\draw [fill=black] (0.2,-0.69) circle (0.08);
\draw [fill=black] (-1.005,-1.28) circle (0.08);
\draw [thick,style=dashed,rotate around={150:(0,0)}] (180:3 and 0.7) arc (180:0:3 and 0.7);
\draw [thick,rotate around={150:(0,0)}] (180:3 and 0.7) arc (180:240:3 and 0.7);
\draw [thick,rotate around={150:(0,0)}] (360:3 and 0.7) arc (360:265:3 and 0.7);
\draw [fill=black,rotate around={120:(0,0)}] (0.2,-0.69) circle (0.08);
\draw [fill=black,rotate around={120:(0,0)}] (-1.005,-1.28) circle (0.08);
\draw [thick,style=dashed,rotate around={270:(0,0)}] (180:3 and 0.7) arc (180:0:3 and 0.7);
\draw [thick,rotate around={270:(0,0)}] (180:3 and 0.7) arc (180:240:3 and 0.7);
\draw [thick,rotate around={270:(0,0)}] (360:3 and 0.7) arc (360:265:3 and 0.7);
\draw [fill=black,rotate around={240:(0,0)}] (0.2,-0.69) circle (0.08);
\draw [fill=black,rotate around={240:(0,0)}] (-1.005,-1.28) circle (0.08);
\draw[thick,rotate around={-9.7:(0,0)}] (57:3 and 0.05) arc (57:105:3 and 0.05);
\draw[thick,rotate around={110.3:(0,0)}] (57:3 and 0.05) arc (57:105:3 and 0.05);
\draw[thick,rotate around={230.3:(0,0)}] (57:3 and 0.05) arc (57:105:3 and 0.05);
\draw (-0.8,1.85) node {$1$};
\draw (-0.95,0) node {$2$};
\draw (1.9,-0.1) node {$3$};
\draw (0.4,0.8) node {$4$};
\draw (-1,-1.65) node {$5$};
\draw (0.5,-0.8) node {$6$};
\end{tikzpicture}
\hspace{1.5cm}
\begin{tikzpicture}
\draw [thick] (0,0) circle (3);
\draw [thick,style=dashed,rotate around={30:(0,0)}] (-3,0) arc (180:0:3 and 0.7);
\draw [thick,rotate around={30:(0,0)}] (-3,0) arc (180:240:3 and 0.7);
\draw [thick,rotate around={30:(0,0)}] (3,0) arc (360:265:3 and 0.7);
\draw [fill=black] (0.2,-0.69) circle (0.08);
\draw [fill=black] (-1.005,-1.28) circle (0.08);
\draw [thick,style=dashed,rotate around={150:(0,0)}] (-3,0) arc (180:0:3 and 0.7);
\draw [thick,rotate around={150:(0,0)}] (-3,0) arc (180:240:3 and 0.7);
\draw [thick,rotate around={150:(0,0)}] (3,0) arc (360:265:3 and 0.7);
\draw [fill=black,rotate around={120:(0,0)}] (0.2,-0.69) circle (0.08);
\draw [fill=black,rotate around={120:(0,0)}] (-1.005,-1.28) circle (0.08);
\draw [thick,style=dashed,rotate around={270:(0,0)}] (-3,0) arc (180:0:3 and 0.7);
\draw [thick,rotate around={270:(0,0)}] (-3,0) arc (180:240:3 and 0.7);
\draw [thick,rotate around={270:(0,0)}] (3,0) arc (360:265:3 and 0.7);
\draw [fill=black,rotate around={240:(0,0)}] (0.2,-0.69) circle (0.08);
\draw [fill=black,rotate around={240:(0,0)}] (-1.005,-1.28) circle (0.08);
\draw[thick,rotate around={-9.7:(0,0)}] (57:3 and 0.05) arc (57:105:3 and 0.05);
\draw[thick,rotate around={230.3:(0,0)}] (57:3 and 0.05) arc (57:105:3 and 0.05);
\draw (-0.85,1.6) node {$1$};
\draw (-0.95,0) node {$2$};
\draw (1.9,-0.1) node {$3$};
\draw (0.4,0.8) node {$4$};
\draw (-1,-1.65) node {$5$};
\draw (0.45,-0.8) node {$6$};
\draw [fill=black] (0.9,-1.0) circle (0.08);
\draw [fill=black] (-0.7,2.3) circle (0.08);
\draw[thick,rotate around={16.8:(0,0)}] (126:0.01 and 3) arc (126:195:0.01 and 3);
\draw[thick,rotate around={30.65:(0,0)}] (-26:0.293 and 3) arc (-26:32:0.293 and 3);
\draw[thick,rotate around={23.7:(0,0)}] (-27:0.465 and 3) arc (-27:52:0.465 and 3);
\draw (-0.95,2.4) node {$7$};
\draw (1.15,-1) node {$8$};
\end{tikzpicture}

\caption{Spherical thrackle drawings of a $6$-cycle (left) and an $8$-cycle (right).}\label{F1}
\end{figure}

The main result of this paper is that Conway's thrackle conjecture (see \cite{W}) holds in the case of spherical thrackles.

\begin{theorem}\label{thcon}
Let $G$ be an abstract graph with $n$ vertices and $m$ edges. If $G$ admits a spherical thrackle drawing, then $n\geq m$.
\end{theorem}

Let us comment straight away that Theorem \ref{thcon} opens up an approach to the thrackle conjecture in the plane; given a thrackle in the plane one can transport it to the sphere by central projection, and then attempt to deform it to a thrackle whose edges are arcs of great circles. We have not been able to complete this program; the difficulty lies in controlling the process so that during the deformation the edges do not cross  any vertices.

\section{Definitions and assumptions}
Consider a spherical thrackle drawing of a graph $G$.
Throughout this paper the graph $G$ will be assumed to be connected and have no terminal edges. This is not restrictive, since the existence of any counterexample to the thrackle conjecture obviously implies the existence of a counterexample which is connected and has no terminal edges.

We define the \emph{crossing orientation} of any two directed edges $e$, $f$ in a similar manner to the vector cross product. To demonstrate this, in Figure \ref{F2} we have $\chi(e_3,e_1)=1$, while $\chi(e_2,e_4)=-1$. A similar definition applies for intersections at endpoints; in Figure \ref{F2} we have $\chi(e_1,e_2)=1$.  Note that in general we have $\chi(e,f)=-\chi(f,e)$.
\begin{figure}[h]
\begin{tikzpicture}
\draw [thick] (0,0) circle (3);
\draw [thick,rotate around={4.6:(0,0)}] (103:3 and 0.79) arc (103:125:3 and 0.79);
\draw [thick,rotate around={4.6:(0,0)},-latex] (60:3 and 0.79) arc (60:105:3 and 0.79);
\draw [thick,rotate around={-46.4:(0,0)}] (113:3 and 0.84) arc (113:128:3 and 0.84);
\draw [thick,rotate around={-46.4:(0,0)},-latex] (55:3 and 0.84) arc (55:115:3 and 0.84);
\draw [thick,rotate around={-20.4:(0,0)}] (258:3 and 0.21) arc (258:310:3 and 0.21);
\draw [thick,rotate around={-20.4:(0,0)},-latex] (230:3 and 0.21) arc (230:260:3 and 0.21);
\draw [thick,rotate around={48.8:(0,0)}] (258:3 and 0.71) arc (258:302:3 and 0.71);
\draw [thick,rotate around={48.8:(0,0)},-latex] (237:3 and 0.71) arc (237:260:3 and 0.71);
\draw [fill=black] (-0.8,1.8) circle (0.08);
\draw [fill=black] (-1.8,0.5) circle (0.08);
\draw [fill=black] (-0.6,-1.6) circle (0.08);
\draw [fill=black] (1.5,0.8) circle (0.08);
\draw [fill=black] (1.7,-0.8) circle (0.08);
\draw (0.3,-1.2) node {$e_1$};
\draw (-1.2,0.9) node {$e_2$};
\draw (-0.8,-0.2) node {$e_3$};
\draw (-0.1,1.6) node {$e_4$};
\end{tikzpicture}

\caption{A directed $4$-path.}\label{F2}
\end{figure}
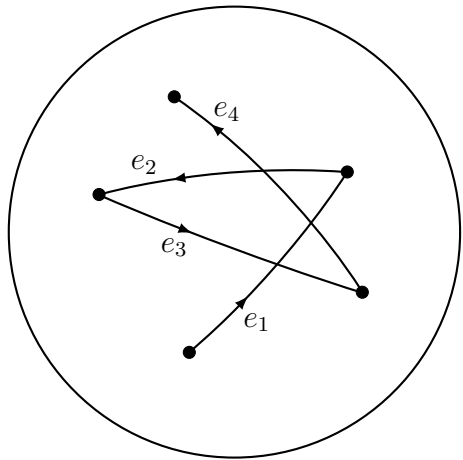

A (directed) $k$-path $p=e_1\dots e_k$ is called \emph{good} if either $\chi(e_{i-1},e_{i})=1$ for each $i=2,\dots,k$ or $\chi(e_{i-1},e_{i})=-1$ for each $i=2,\dots,k$, and is called \emph{bad} otherwise. Similarly, a $k$-cycle $c_k$ is called \emph{good} if every directed path in $c_k$ is good, and is called bad otherwise. The path shown in Figure \ref{F2} is good.

For any edge $e$, denote by ${\C}(e)$ the great circle containing $e$.
A \emph{long} (resp.~\emph{short}, resp.~\emph{medium}) edge is an edge whose length is greater than $\pi$ (resp.~less than $\pi$, resp.~equal to $\pi$).

For a given spherical thrackle drawing, by making small adjustments to the vertex positions if necessary,  we may assume  that no two edges lie on the same great circle. Hence the crossing orientations are all well-defined. Similarly, we may also assume that there are no medium edges, so every edge is either long or short. With these assumptions we say that the drawing is in \emph{general position}. Throughout this paper, the term \emph{spherical thrackle} will designate a spherical thrackle drawing in general position.

We will say that an edge $e$ \emph{lies in} a given (open) hemisphere $\mathcal{H}$ if the interior of $e$ is contained in $\mathcal{H}$; that is, we permit the vertices of $e$ to lie on the boundary of $\mathcal{H}$.
We require one further notion. Let $v\in G$ be a vertex of degree $\geq 3$ and let $e$ be an edge incident to $v$. For the purpose of illustration, suppose any other edges incident to $v$ are directed away from $v$. We say $e$ \emph{separates at} $v$ if and only if there are two other edges, $f,g$, incident to $v$  which start in opposite hemispheres bounded by ${\C}
(e)$. This is illustrated in Figure \ref{F4}. If it is understood which vertex is being referred to, or it is irrelevant, we simply say $e$ \emph{separates}.

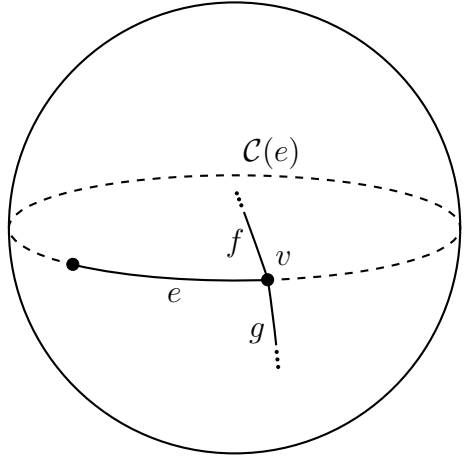
\begin{figure}[h!]
\begin{tikzpicture}
\draw [thick] (0,0) circle (3);
\draw [thick] (280:3 and 0.7) arc (280:225:3 and 0.7);
\draw [thick,style=dashed] (225:3 and 0.7) arc (225:-80:3 and 0.7);
\draw [fill=black] (0.44,-0.69) circle (0.08);
\draw [fill=black] (.09,0.3) circle (0.02);
\draw [fill=black] (.06,0.38) circle (0.02);
\draw [fill=black] (.03,0.46) circle (0.02);
\draw [fill=black] (.56,-1.65) circle (0.02);
\draw [fill=black] (.57,-1.75) circle (0.02);
\draw [fill=black] (.58,-1.85) circle (0.02);
\draw [thick,rotate around={110:(0,0)}] (255:3 and 0.19) arc (255:273:3 and 0.19);
\draw [fill=black] (-2.15,-0.485) circle (0.08);
\draw [thick,rotate around={-80:(0,0)}] (75:3 and 0.327) arc (75:57:3 and 0.327);
\draw (0.01,-0.2) node {$f$};
\draw (0.3,-1.4) node {$g$};
\draw (0.65,-0.4) node {$v$};
\draw (-0.8,-0.9) node {$e$};
\draw (0.5,1) node {${\C}
(e)$};
\end{tikzpicture}

\caption{The edge $e$ separates at $v$.}\label{F4}
\end{figure}

\section{Cycles}

\begin{theorem}\label{cth} Every spherical thrackle drawing of an $n$-cycle is good for $n\geq5$. Moreover, if $n$ is even, such a drawing contains at least one long edge.
\end{theorem}

\begin{proof} Let $c_n=e_1e_2\dots e_n$ be an $n$-cycle for some $n\geq 5$. For convenience of notation, let us write $e_k=e_{k+n}$ for all $k$, and choose the direction on $c_n$ in order of increasing edge index.

Suppose $c_n$ is bad. We can assume without loss of generality that there are three adjacent edges $e_{j-1}, e_j, e_{j-1}$ such that $\chi(e_{j-1},e_{j})=1$ and $\chi(e_{j},e_{j+1})=-1$. Then $e_j$ is short; indeed, otherwise $e_{j-1}$ and $e_{j+1}$ will be forced to lie in opposite hemispheres bounded by ${\C}
(e_j)$ and hence could not intersect. We must also have at least one of $e_{j-1}$ and $e_{j+1}$ long, for the same reason.
Up to relabelling and direction change, we have the structure shown in Figure \ref{F3}, with $e_{j+1}$ possibly long. We see the edge $e_{j-2}$ incident to $e_{j-1}$ at its starting point must meet $e_{j}$ and $e_{j+1}$ at their common endpoint in order to intersect them both, and this produces a $3$-cycle, which is a contradiction. Hence, $c_n$ is good.

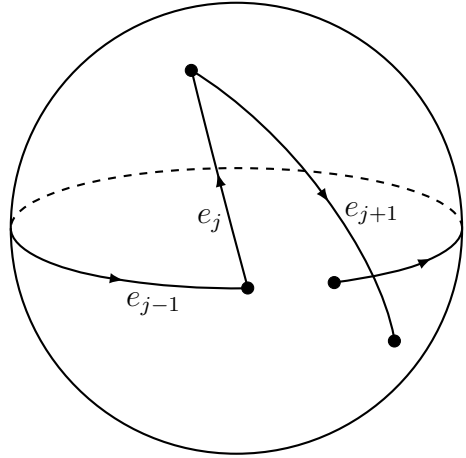
\begin{figure}
\begin{tikzpicture}
\draw [thick] (0,0) circle (3);
\draw [thick] (238:3 and 0.8) arc (238:273:3 and 0.8);
\draw [thick,-latex] (180:3 and 0.8) arc (180:240:3 and 0.8);
\draw [thick,style=dashed] (0:3 and 0.8) arc (0:180:3 and 0.8);
\draw [thick] (328:3 and 0.8) arc (328:360:3 and 0.8);
\draw [thick,-latex] (295:3 and 0.8) arc (295:330:3 and 0.8);
\draw [thick,rotate around={14.8:(0,0)}] (167:0.058 and 3) arc (167:135:0.058 and 3);
\draw [thick,rotate around={14.8:(0,0)},-latex] (195:0.058 and 3) arc (195:165:0.058 and 3);
\draw [thick,rotate around={39.8:(0,0)}] (-8:1.185 and 3) arc (-8:-55:1.185 and 3);
\draw [thick,rotate around={39.8:(0,0)},-latex] (42:1.185 and 3) arc (42:-10:1.185 and 3);
\draw [fill=black] (0.15,-0.795) circle (0.08);
\draw [fill=black] (-0.6,2.1) circle (0.08);
\draw [fill=black] (2.1,-1.5) circle (0.08);
\draw [fill=black] (1.3,-0.723) circle (0.08);
\draw (-1.1,-1.0) node {$e_{j-1}$};
\draw (-0.35,0.1) node {$e_j$};
\draw (1.8,0.2) node {$e_{j+1}$};
\end{tikzpicture}

\caption{A bad $3$-path.}\label{F3}
\end{figure}

Now, let $c_{2n}=e_1\dots e_{2n}$ be a $2n$-cycle for some $n\geq3$, directed in order of increasing edge index. Suppose that all edges in $c_{2n}$ are short. Denote by $\mathcal{H}$ the hemisphere bounded by ${\C}
(e_1)$ in which the edge $e_2$ lies. As $c_{2n}$ is good, the starting point of $e_{2n}$ is also in $\mathcal{H}$. As $e_3$ is short and begins in $\mathcal{H}$, it ends in the other hemisphere $-\mathcal{H}$ (since it must cross $e_1$ and hence ${\C}
(e_1)$). By induction,  each even-numbered edge (other than $e_{2n}$) ends in $\mathcal{H}$, while each odd-numbered edge (other than $e_1$) ends in $-\mathcal{H}$. But $e_{2n-1}$ must end in $\mathcal{H}$ in order to meet the starting point of $e_{2n}$, so we have a contradiction. Hence, $c_{2n}$ contains a long edge.
\end{proof}

\begin{remark}\label{R:bad3cycle}
Notice that since we will be dealing with graphs without terminal edges, the proof of Theorem \ref{cth} also establishes the following fact: for every bad $3$-path, the middle edge is short and is an edge in a bad $3$-cycle. Moreover, every bad $3$-cycle has a unique long edge; this is the edge whose vertices have the same crossing orientation. These observations will be useful in what follows.
\end{remark}

\section{Lemmata}

Consider a spherical thrackle drawing of a connected graph $G$  having no terminal edges.

\begin{lemma}\label{sel}
In any spherical thrackle, every edge that separates at one of its vertices is short and is an edge of a bad $3$-cycle.
\end{lemma}

\begin{proof} Assume that $e$ separates at  $v$ with adjacent edges $f,g$ in different hemispheres. Direct $f,g$ towards $v$ and $e$ away from $v$, and assume without loss of generality that $\chi(f,e)=-1$ and $\chi(g,e)=1$,  as shown in Figure \ref{F4}. Since $G$ has no terminal edges, there is some edge $h$ incident to $e$ at its other endpoint. We have either $\chi(e,h)=1$ or $\chi(e,h)=-1$. If $\chi(e,h)=1$, then $feh$ is a bad $3$-path, and if $\chi(e,h)=-1$, then $geh$ is a bad $3$-path, so in either case $e$ is the middle edge of a bad $3$-path. So by Remark \ref{R:bad3cycle}, $e$ is short and is an edge in a bad $3$-cycle.
\end{proof}

\begin{lemma}\label{vl} Let $G$ be a connected graph with no terminal edges. If $\deg(v)\geq3$, then there exists a great circle ${\C}
$ passing through $v$ such that the starting segments of all of the edges incident to $v$ lie in the same hemisphere bounded by ${\C}
$.
\end{lemma}

\begin{proof} If no such circle ${\C}
$ exists, then all of the edges incident to $v$ must separate; since there are at least three of them, this gives at least two different $3$-cycles. But this is impossible as thrackles can have at most one $3$-cycle \cite{CN}.
\end{proof}

\begin{lemma}\label{corollary} For any vertex $v$ we have  $\deg(v)\leq4$. Moreover, if $\deg(v)>2$, then $v$ is a vertex of a bad $3$-cycle.
\end{lemma}

\begin{proof} By Lemma \ref{vl}, if we have a vertex of degree $5$, then at least three of the adjacent edges are separating edges; this gives at least two different $3$-cycles by Lemma \ref{sel}, which is impossible. Hence, $\deg(v)\leq4$ for any vertex $v$. If $\deg(v)>2$, we have at least one separating edge, which must be an edge of a bad $3$-cycle by Lemma \ref{sel}, so $v$ is a vertex of a bad $3$-cycle.
\end{proof}

\begin{lemma}\label{l:2long}{\ }
\begin{enumerate}[{\rm (a)}]
\item No two long edges are adjacent.
\item  All the edges adjacent to a long edge $e$  lie in the
same open hemisphere bounded by ${\C}(e)$.
\item  If $e_1e_2e_3$ is a directed $3$-path, with the edge $e_2$ long, then
the orientations of crossings of $e_2$ and $e_3$ with $e_1$ are the same:
$\chi (e_1,e_2)=\chi (e_1,e_3)$.
\item  If $e_1e_2$ is a directed $2$-path, with both edges $e_1, e_2$ short, then
any directed edge $e$ not incident to their common endpoint crosses $e_1, e_2$
with opposite orientations:  $\chi (e,e_2)=-\chi (e,e_1)$.
\item  Let $p = e_1e_2\ldots e_m$ be a directed path, with all the edges
short, and let $e$ be a directed edge none of whose endpoints is a common endpoint
of two adjacent edges of $p$. Then $\chi (e, e_m)=(-1)^{m-1} \chi (e, e_1)$.
\end{enumerate}
\end{lemma}

\begin{proof} (a) is  trivial, as two adjacent long edges would have two common
points.

(b) Two edges that are adjacent to $e$ at different endpoints of $e$ are short by (a) and lie in the
same hemisphere bounded by ${\C}(e)$ since otherwise they would have no points in common.
The assertion follows. Part (c) is obvious.

(d) follows from the fact that both edges $e_1, e_2$ are short, hence each cross or meet the
great circle ${\C}(e)$ exactly once.
Part (e) follows from (d) by induction.
\end{proof}

\begin{lemma}\label{gpl} Let $e_0e_1e_2\dots e_{m-1}e_{m}e_{m+1}$ be a simple good path with $e_1, e_m$ long and all other edges short. Then $m$ is odd; that is, the long edges are separated by an odd number of short edges.
\end{lemma}

\begin{proof} Assume $m$ is even and direct the path from $e_0$ to $e_{m+1}$. By Lemma~\ref{l:2long}(a), $m \ge 4$, so
$e_2 \ne e_{m-1}$. Without loss of generality, assume that $\chi (e_i, e_{i+1})=1$ for all $i \le m$.
The proof has three steps.

\emph{Step 1.} We first compute the orientations of some crossings.
By Lemma \ref{l:2long}(c), $\chi (e_0, e_2) = \chi (e_0, e_1) = 1$. Applying
Lemma~\ref{l:2long}(e) to the $(m-2)$-path $e_2 \ldots e_{m-1}$ and the edges $e_0$ and $e_1$ we obtain respectively
$\chi (e_0, e_{m-1}) = (-1)^{m-3}\chi (e_0, e_{2}) = -1$ and
$\chi (e_1, e_{m-1}) = (-1)^{m-3}\chi (e_1, e_{2})= -1$. Similarly,
$\chi (e_{m+1}, e_2) = \chi (e_m, e_2) =1$. Again, by Lemma~\ref{l:2long}(e),
$\chi (e_2, e_{m-1})= \chi (e_2, e_{3})= 1$.

\emph{Step 2.} We claim that on the edge $e_1$, the crossing point $e_1 \cap e_{m-1}$ appears before
the crossing point $e_1 \cap e_m$. Similarly, on the edge $e_m$, the crossing point $e_m \cap e_2$
appears after the crossing point $e_m \cap e_1$.

To see this, we only need the edges $e_1, e_2, e_{m-1}, e_m$ and the information about the orientation of
their crossings obtained in Step 1:
\[
\chi (e_1, e_2)= \chi (e_{m-1}, e_m)=1,\quad \chi (e_1, e_{m-1}) = \chi (e_2, e_m) = -1,\quad \chi (e_2, e_{m-1})=1.
\]
Let $\mathcal{H}$ be the open hemisphere bounded by ${\C}(e_1)$ in which the short
edge $e_2$ lies. Using the orientation of mutual crossings of $e_1,e_2$ and $e_{m-1}$, we find that the
short edge $e_{m-1}$ starts in $\mathcal{H}$ and ends in $\mathcal{-H}$, the antipodal hemisphere. Then
the long edge $e_m$ starts in $\mathcal{-H}$. If it crossed $e_2$ before $e_1$, the orientation
$\chi  (e_m, e_2)$ would be $-1$, which is not the case. Thus, on the edge $e_m$, the crossing with
$e_2$ is after the crossing with $e_1$. Changing the direction of the path and relabelling edges
accordingly, we find that on the edge $e_1$, the crossing with $e_{m-1}$ is before the crossing with
$e_m$.

\emph{Step 3.} Let $v_i$ be the starting point of the edge $e_i$ for $ i =0, \dots, m+1$, and let
$p=e_1 \cap e_m$, $q=e_2 \cap e_{m+1}$ and $r=e_2 \cap e_{m}$. As we found in Step 1, $\chi (e_{m+1}, e_2)= \chi (e_m, e_2)= 1$. Consequently, since $v_{m+1}$
is the starting point of the short edge $e_{m+1}$, the arc $f$ of the edge $e_m$ from $r$ to $v_{m+1}$ is longer than $\pi$. Indeed, if $g$ denotes the arc of $e_2$ travelled in the reverse direction  from $q$ to $r$, and $h$ denotes the arc of $e_{m+1}$ from $v_{m+1}$ to $q$, then the triangle $gfh$ is a bad $3$-cycle and thus $f$ has length greater than $\pi$ by Remark \ref{R:bad3cycle}. Then, by Step 2, the arc $pv_{m+1}$
of the edge $e_m$ is longer than $\pi$. Similarly, the arc $v_1p$ of the edge $e_1$ is longer than
$\pi$. But then the edges $e_m$ and $e_1$ have two points in common, which is a contradiction.
\end{proof}

\begin{lemma}\label{gcy} In every good cycle, consecutive distinct long edges are separated by an odd number of short edges.
In particular, good odd cycles have at most one long edge.
 \end{lemma}

\begin{proof} By Lemma \ref{l:2long}(a), there are no adjacent long edges. If $e_k$ and $e_{\ell}$ are long edges with $k< \ell$ and there are no long edges $e_i$ for $k<i< \ell$, then one can consider the path $e_{k-1}e_k\dots e_{\ell}e_{\ell+1}$. So the first part of the lemma follows immediately from Lemma \ref{gpl} except in the case of a good cycle of the form $e_0e_1e_2\dots e_{m-1}e_{m}$ where $e_1, e_m$ are long and all other edges short. But in this case one can replace $e_0$ by two edges $e'_0,e''_0$, as in Figure \ref{Fsw}, and apply
Lemma \ref{gpl} to the good path $e'_0e_1e_2\dots e_{m-1}e_{m}e''_0$.

\begin{figure}
\begin{tikzpicture}
\draw[thick,dashed] (-1.5,1.5) -- (-1,1);
\draw[thick,dashed] (4.5,1.5) -- (4,1);
\draw[thick,-] (-1,1) -- (0,0) -- (3,0) -- (4,1);
 \draw [fill=black] (0,0) circle (0.08);
 \draw [fill=black] (3,0) circle (0.08);
\draw (1.5,-.3) node {$e_0$};
\draw (4,.6) node {$e_1$};
\draw (-1,.6) node {$e_m$};
\end{tikzpicture}
\hspace{1.5cm}
\begin{tikzpicture}
\draw[thick,dashed] (-1.5,1.5) -- (-1,1);
\draw[thick,dashed] (4.5,1.5) -- (4,1);
\draw[thick,-] (-1,1) -- (0,0) -- (3,-0.5);
\draw[thick,-] (0,-0.5) -- (3,0) -- (4,1);
 \draw [fill=black] (0,-0.5) circle (0.08);
 \draw [fill=black] (3,-0.5) circle (0.08);
 \draw [fill=black] (0,0) circle (0.08);
 \draw [fill=black] (3,0) circle (0.08);
\draw (.6,.15) node {$e''_0$};
\draw (2.3,.15) node {$e'_0$};
\draw (4,.6) node {$e_1$};
\draw (-1,.6) node {$e_m$};
\end{tikzpicture}
\caption{Turning a good cycle into a good path.}\label{Fsw}
\end{figure}
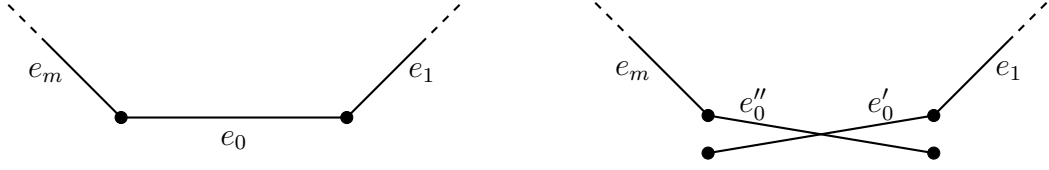

If a good cycle has $n$ long edges with $n\ge2$, then from what we have just seen there is an odd number of short edges between each consecutive pair of long edges. So the total number $t$  of short edges is even if and only if $n$ is even. But then the total number of edges is $t+n$, which is always even.
\end{proof}

\section{Proof of the main result}

Consider a spherical thrackle drawing and suppose, by way of contradiction, that $G$ has more edges than vertices.
We prove Theorem \ref{thcon} in two parts. We first prove the existence of a spherical thrackle drawing of a figure $8$-graph $H$; that is,  $H$ consists of two cycles that only share a vertex. In particular, we will show that $H$ consists of an even cycle and a bad $3$-cycle, with the vertex of degree $4$ opposite the long edge of the bad $3$-cycle. We then prove that no such graph can be drawn as a spherical thrackle.

If $G$ has more edges than vertices, then there must be some vertex in $G$ with degree greater than $2$. By Lemma \ref{corollary}, any such vertex is the vertex of a bad $3$-cycle $c_3$ of $G$. Since $G$ can contain at most one $3$-cycle, only the vertices of $c_3$ can have degree greater than $2$.

Let $c_3=ABC$, with $A$ opposite to the unique long edge $BC$; see Remark \ref{R:bad3cycle}. Let $\mathcal{H}$ be the hemisphere bounded by ${\C}
(BC)$ containing $A$. We consider the three possible cases (up to symmetry):
\begin{enumerate}
\item $\deg(B)>2$ and $\deg(C)>2$.
\item $\deg(B)>2$ and $\deg(C)=2$.
\item $\deg(B)=\deg(C)=2$.
\end{enumerate}

Case (1). As $\deg(B)>2$,  there is some edge $BD$ incident to $B$ with $D\neq A,C$. Since $BC$ is long, $BD$ is necessarily short by Lemma~\ref{l:2long}(a), so $D\in\mathcal{H}$ since $BD$ must cross $AC$. $BD$ does not separate, since otherwise it produces another bad $3$-cycle by Lemma \ref{sel}. If there is some other edge $BD'$ incident to $B$ with $D'\neq A,C,D$, then one of $BD$, $BD'$ would separate, giving another bad $3$-cycle. Hence $\deg(B)=3$. By symmetry, $\deg(C)=3$, with some short edge $CE$ incident to $C$ with $E\in\mathcal{H}$. We must then have $\deg(A)=2$, since if there were some edge $AF$ incident to $A$ with $F\neq B,C$ then $AF$ cannot separate $AB$, $AC$, or it produces another bad $3$-cycle, but this would imply that $AF$ has no points in common with one of $BD$, $CE$.

This brings us  to the drawing on the left in Figure \ref{F5}, from which we can obtain the structure shown on the right in Figure \ref{F5} by edge insertion. Since all other vertices have degree $2$, we have a cycle sharing the vertex $A$ with a $3$-cycle. Since the intersection of the two cycles is a touching intersection, the other cycle must be even by \cite{LPS}.

\begin{figure}
\begin{tikzpicture}
\draw [thick] (0,0) circle (3);
\draw [thick,style=dashed] (-3,0) arc (180:0:3 and 0.7);
\draw [thick] (-3,0) arc (180:285:3 and 0.7);
\draw [thick] (3,0) arc (360:315:3 and 0.7);
\draw [thick] (-13.2:0.8 and 3) arc (-13.2:72:0.8 and 3);
\draw [thick,rotate around={12:(0,0)}] (-18.2:2.07 and 3) arc (-18.2:66:2.07 and 3);
\draw [thick,rotate around={50:(0,0)}] (269.5:3 and 1.03) arc (269.5:340:3 and 1.03);
\draw [thick,rotate around={-18.65:(0,0)}] (44:3 and 0.3) arc (44:98:3 and 0.3);
\draw [fill=black] (0.26,2.85) circle (0.08);
\draw (-0.05,2.8) node {$A$};
\draw [fill=black] (0.78,-0.665) circle (0.08);
\draw (1,-0.9) node {$C$};
\draw [fill=black] (2.117,-0.486) circle (0.08);
\draw (2.05,-0.8) node {$B$};
\draw [fill=black] (2.084,1.93) circle (0.08);
\draw (2.25,1.6) node {$E$};
\draw [fill=black] (-0.293,0.41) circle (0.08);
\draw (-0.5,0.1) node {$D$};
\end{tikzpicture}
\hspace{1.5cm}
\begin{tikzpicture}
\draw [thick] (0,0) circle (3);
\draw [thick,style=dashed] (-3,0) arc (180:0:3 and 0.7);
\draw [thick] (-3,0) arc (180:285:3 and 0.7);
\draw [thick] (3,0) arc (360:315:3 and 0.7);
\draw [thick] (-13.2:0.8 and 3) arc (-13.2:72:0.8 and 3);
\draw [thick,rotate around={12:(0,0)}] (-18.2:2.07 and 3) arc (-18.2:66:2.07 and 3);
\draw [thick,rotate around={46.7:(0,0)}] (250:3 and 0.6) arc (250:342:3 and 0.6);
\draw [thick,rotate around={-6.15:(0,0)}] (108:0.175 and 3) arc (108:202.7:0.175 and 3);
\draw [thick,rotate around={-12.8:(0,0)}] (21:3 and 0.33) arc (21:100:3 and 0.33);
\draw [thick,rotate around={32:(0,0)}] (-40:2.66 and 3) arc (-40:50:2.66 and 3);
\draw [fill=black] (0.26,2.85) circle (0.08);
\draw (-0.05,2.8) node {$A$};
\draw [fill=black] (0.78,-0.665) circle (0.08);
\draw (1,-0.9) node {$C$};
\draw [fill=black] (-0.285,-1.13) circle (0.08);
\draw (0.05,-1.2) node {$C'$};
\draw [fill=black] (2.117,-0.486) circle (0.08);
\draw (2.05,-0.8) node {$B$};
\draw [fill=black] (2.75,-0.5) circle (0.08);
\draw (2.6,-0.85) node {$B'$};
\draw [fill=black] (2.084,1.93) circle (0.08);
\draw (2.4,2.2) node {$E$};
\draw [fill=black] (-0.425,0.432) circle (0.08);
\draw (-0.6,0.11) node {$D$};
\end{tikzpicture}

\caption{The  drawing on the right is obtained from the drawing on the left.}\label{F5}
\end{figure}

Case (2). As in case (1), $\deg(B)=3$, with some short edge $BD$ incident to $B$ crossing $AC$. Since $\deg(C)=2$, and all vertices other than the vertices of $c_3$ have degree $2$, we must have $\deg(A)$ odd so that the degree sum of $G$ is even (which is true of any graph).  Hence, $\deg(A)=3$, since this is the only possible odd degree by Lemma \ref{corollary}. The remaining edge $AF$ (with $F\neq B,C$) incident to $A$ must intersect $BD$ and $BC$, and cannot separate $AB$, $AC$, so we get the drawing on the left in Figure \ref{F6}. By edge insertion as shown on the right in Figure \ref{F6}, we again obtain a figure $8$-graph consisting of an even cycle and a $3$-cycle.

\begin{figure}
\begin{tikzpicture}
\draw [thick] (0,0) circle (3);
\draw [thick,style=dashed] (-3,0) arc (180:0:3 and 0.7);
\draw [thick] (-3,0) arc (180:285:3 and 0.7);
\draw [thick] (3,0) arc (360:315:3 and 0.7);
\draw [thick] (-13.2:0.8 and 3) arc (-13.2:72:0.8 and 3);
\draw [thick,rotate around={12:(0,0)}] (-18.2:2.07 and 3) arc (-18.2:66:2.07 and 3);
\draw [thick,rotate around={-18.65:(0,0)}] (44:3 and 0.3) arc (44:98:3 and 0.3);
\draw [thick,rotate around={-6.15:(0,0)}] (108:0.175 and 3) arc (108:202.7:0.175 and 3);
\draw [fill=black] (0.26,2.85) circle (0.08);
\draw (-0.05,2.8) node {$A$};
\draw [fill=black] (0.78,-0.665) circle (0.08);
\draw (1,-0.9) node {$C$};
\draw [fill=black] (2.117,-0.486) circle (0.08);
\draw (2.05,-0.8) node {$B$};
\draw [fill=black] (-0.293,0.41) circle (0.08);
\draw (-0.5,0.1) node {$D$};
\draw [fill=black] (-0.285,-1.13) circle (0.08);
\draw (0.05,-1.2) node {$F$};

\end{tikzpicture}
\hspace{1.5cm}
\begin{tikzpicture}
\draw [thick] (0,0) circle (3);
\draw [thick,style=dashed] (-3,0) arc (180:0:3 and 0.7);
\draw [thick] (-3,0) arc (180:285:3 and 0.7);
\draw [thick] (3,0) arc (360:315:3 and 0.7);
\draw [thick] (-13.2:0.8 and 3) arc (-13.2:72:0.8 and 3);
\draw [thick,rotate around={12:(0,0)}] (-18.2:2.07 and 3) arc (-18.2:66:2.07 and 3);
\draw [thick,rotate around={-6.15:(0,0)}] (108:0.175 and 3) arc (108:202.7:0.175 and 3);
\draw [thick,rotate around={-12.8:(0,0)}] (21:3 and 0.33) arc (21:100:3 and 0.33);
\draw [thick,rotate around={32:(0,0)}] (-40:2.66 and 3) arc (-40:50:2.66 and 3);
\draw [fill=black] (0.26,2.85) circle (0.08);
\draw (-0.05,2.8) node {$A$};
\draw [fill=black] (0.78,-0.665) circle (0.08);
\draw (1,-0.9) node {$C$};
\draw [fill=black] (-0.285,-1.13) circle (0.08);
\draw (0.05,-1.2) node {$F$};
\draw [fill=black] (2.117,-0.486) circle (0.08);
\draw (2.05,-0.8) node {$B$};
\draw [fill=black] (2.75,-0.5) circle (0.08);
\draw (2.6,-0.85) node {$B'$};
\draw [fill=black] (-0.425,0.432) circle (0.08);
\draw (-0.6,0.11) node {$D$};
\end{tikzpicture}

\caption{The drawing on the right is obtained from the drawing on the left.}\label{F6}
\end{figure}

Case (3). Since $A$ is the only vertex with degree greater than $2$, and the degree sum must be even, we have $\deg(A)=4$ by Lemma \ref{corollary}. The other two edges $AF$, $AF'$ must both cross $BC$, and since neither of them can separate at $A$, we get (without loss of generality) the structure shown in Figure \ref{F7}. We again have an even cycle sharing a vertex with a $3$-cycle.

\begin{figure}
\begin{tikzpicture}
\draw [thick] (0,0) circle (3);
\draw [thick,style=dashed] (-3,0) arc (180:0:3 and 0.7);
\draw [thick] (-3,0) arc (180:285:3 and 0.7);
\draw [thick] (3,0) arc (360:315:3 and 0.7);
\draw [thick] (-13.2:0.8 and 3) arc (-13.2:72:0.8 and 3);
\draw [thick,rotate around={12:(0,0)}] (-18.2:2.07 and 3) arc (-18.2:66:2.07 and 3);
\draw [thick,rotate around={19:(0,0)}] (-46:2.4 and 3) arc (-46:61:2.4 and 3);
\draw [thick,rotate around={-6.15:(0,0)}] (108:0.175 and 3) arc (108:202.7:0.175 and 3);
\draw [fill=black] (0.26,2.85) circle (0.08);
\draw (-0.05,2.8) node {$A$};
\draw [fill=black] (0.78,-0.665) circle (0.08);
\draw (1,-0.9) node {$C$};
\draw [fill=black] (-0.285,-1.13) circle (0.08);
\draw (0.05,-1.2) node {$F'$};
\draw [fill=black] (2.117,-0.486) circle (0.08);
\draw (2.05,-0.8) node {$B$};
\draw [fill=black] (2.28,-1.49) circle (0.08);
\draw (1.95,-1.6) node {$F$};
\end{tikzpicture}

\caption{Case (3) gives another figure $8$-graph.}\label{F7}
\end{figure}
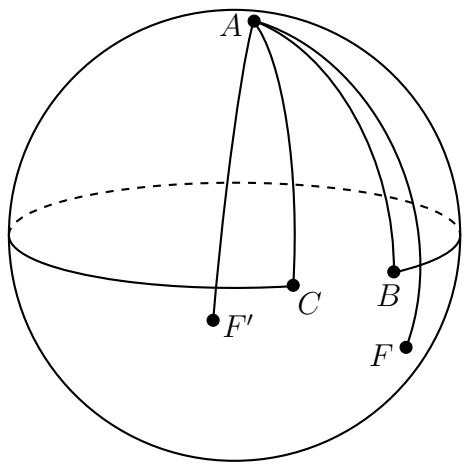

This completes the first part of the proof; in each case we have obtained a figure $8$-graph $H$ consisting of an even cycle and a $3$-cycle, with the vertex of degree $4$ opposite the long edge of the bad $3$-cycle. Now we show that such a graph cannot be drawn in this way.

Retain the labelling of the bad $3$-cycle with vertices $A$, $B$, $C$, with $A$ opposite the long edge $BC$. Let $c_{2n}=e_1\dots e_{2n}$ be the even cycle. Note that $c_{2n}$ is good by Theorem \ref{cth}. Orient the direction on $c_{2n}$ in order of increasing edge index with $e_1$ starting at $A$, and without loss of generality, assume that $\chi(e_{i-1},e_{i})=1$ for each $i=2,\dots,2n$, and $\chi(e_{2n},e_{1})=1$. By swapping the labels of $B$ and $C$ if necessary, assume also that $\chi (AB,BC)=\chi( BC,CA)=1$. Since neither $e_1$ nor $e_{2n}$ can separate at $A$, we then get the drawing on left hand side of Figure~\ref{F8}. Note that necessarily $\chi(e_{2n},AB)=\chi (CA,e_1)=1$.
 In particular, the path $ABCAe_1\dots e_{2n}$ is good.
Note that  $ABCAe_1\dots e_{2n}$ can be made simple by introducing a new vertex $A'$ and disconnecting the edges $CA$ and $e_{1}$ from $A$, as shown in the right hand side of Figure~\ref{F8}. This turns the figure $8$-graph $H$ into the good $(2n+3)$-cycle $ABCA'e'_1e_2\dots e_{2n}$.
By Theorem \ref{cth}, $c_{2n}$ contains at least one long edge. So, since $BC$ is long, the good odd cycle $ABCA'e'_1e_2\dots e_{2n}$ has at least two long edges. But this is impossible by Lemma \ref{gcy}.
This completes the proof of Theorem \ref{thcon}.

\begin{figure}
\begin{tikzpicture}
\draw [thick] (0,0) circle (3);
\draw [thick,style=dashed] (-3,0) arc (180:0:3 and 0.7);
\draw [thick] (-3,0) arc (180:285:3 and 0.7);
\draw [thick] (3,0) arc (360:315:3 and 0.7);
\draw [thick] (-13.2:0.8 and 3) arc (-13.2:72:0.8 and 3);
\draw [thick,rotate around={12:(0,0)}] (-18.2:2.07 and 3) arc (-18.2:66:2.07 and 3);
\draw [thick,rotate around={19:(0,0)}] (-46:2.4 and 3) arc (-46:61:2.4 and 3);
\draw [thick,rotate around={-6.15:(0,0)}] (108:0.175 and 3) arc (108:202.7:0.175 and 3);
\draw [fill=black] (0.26,2.85) circle (0.08);
\draw (-0.05,2.8) node {$A$};
\draw [fill=black] (0.78,-0.665) circle (0.08);
\draw (1,-0.9) node {$C$};
\draw [fill=black] (-0.285,-1.13) circle (0.08);
\draw (-0.3,1.3) node {$e_{1}$};
\draw [fill=black] (2.117,-0.486) circle (0.08);
\draw (2.05,-0.8) node {$B$};
\draw [fill=black] (2.28,-1.49) circle (0.08);
\draw (2.35,1.4) node {$e_{2n}$};
\end{tikzpicture}
\hspace{1.5cm}
\begin{tikzpicture}
\draw [thick] (0,0) circle (3);
\draw [thick,style=dashed] (-3,0) arc (180:0:3 and 0.7);
\draw [thick] (-3,0) arc (180:285:3 and 0.7);
\draw [thick] (3,0) arc (360:315:3 and 0.7);
\draw [thick] (-13.2:0.8 and 3) arc (-23.2:52:0.8 and 3);
\draw [thick,rotate around={19:(0,0)}] (-24.2:2.07 and 2.7) arc (-20.2:66:2.07 and 3);
\draw [thick,rotate around={19:(0,0)}] (-50:2.4 and 3) arc (-46:61:2.4 and 3);
\draw [thick,rotate around={-11.15:(0,0)}] (100:0.175 and 3) arc (108:202.7:0.175 and 3);
\draw [fill=black] (-0.04,2.75) circle (0.08);
\draw [fill=black] (0.55,2.85) circle (0.08);
\draw (-0.35,2.8) node {$A$};
\draw (0.88,2.64) node {$A'$};
\draw [fill=black] (0.78,-0.665) circle (0.08);
\draw (1,-0.9) node {$C$};
\draw [fill=black] (-0.365,-1.03) circle (0.08);
\draw (-0.15,1.3) node {$e'_{1}$};
\draw [fill=black] (2.117,-0.486) circle (0.08);
\draw (2.05,-0.8) node {$B$};
\draw [fill=black] (2.23,-1.62) circle (0.08);
\draw (2.25,1.4) node {$e_{2n}$};
\end{tikzpicture}

\caption{Turning a figure $8$-graph into a cycle.}\label{F8}
\end{figure}
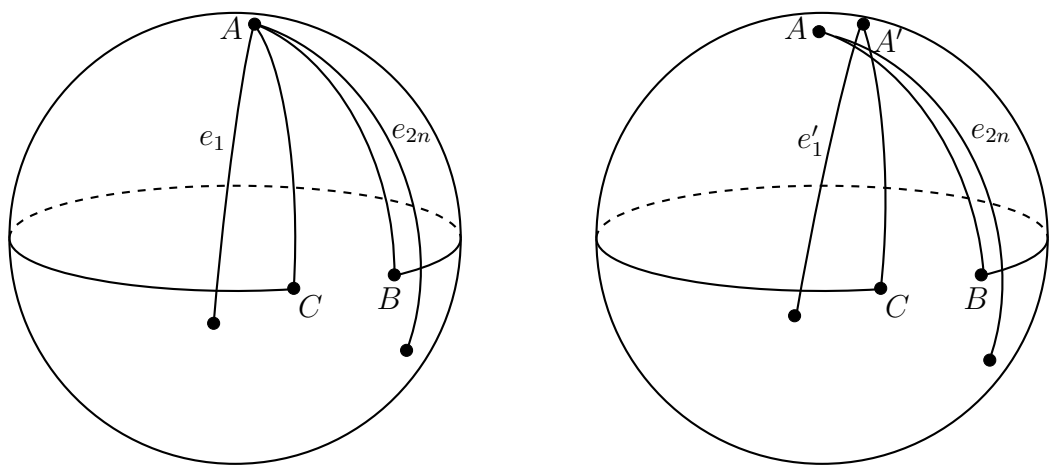

\bibliographystyle{amsplain}
\bibliography{sp}

\providecommand{\bysame}{\leavevmode\hbox to3em{\hrulefill}\thinspace}
\providecommand{\MR}{\relax\ifhmode\unskip\space\fi MR }
\providecommand{\MRhref}[2]{%
  \href{http://www.ams.org/mathscinet-getitem?mr=#1}{#2}
}
\providecommand{\href}[2]{#2}
\begin{thebibliography}{1}

\bibitem{CN}
Grant Cairns and Yuri Nikolayevsky, \emph{Bounds for generalized thrackles},
  Discrete Comput. Geom. \textbf{23} (2000), no.~2, 191--206.

\bibitem{CN2}
\bysame, \emph{Generalized thrackle drawings of non-bipartite graphs}, Discrete
  Comput. Geom. \textbf{41} (2009), no.~1, 119--134.

\bibitem{CN3}
\bysame, \emph{Outerplanar thrackles}, Graphs Combin. \textbf{28} (2012),
  no.~1, 85--96.

\bibitem{FP}
Radoslav Fulek and J{\'a}nos Pach, \emph{A computational approach to {C}onway's
  thrackle conjecture}, Comput. Geom. \textbf{44} (2011), no.~6-7, 345--355.

\bibitem{LPS}
L.~Lov{\'a}sz, J.~Pach, and M.~Szegedy, \emph{On {C}onway's thrackle
  conjecture}, Discrete Comput. Geom. \textbf{18} (1997), no.~4, 369--376.

\bibitem{PRT}
J{\'a}nos Pach, Rado{\v{s}} Radoi{\v{c}}i{\'c}, and G{\'e}za T{\'o}th,
  \emph{Tangled thrackles}, Geombinatorics \textbf{21} (2012), no.~4, 157--169.

\bibitem{PS}
J{\'a}nos Pach and Ethan Sterling, \emph{Conway's conjecture for monotone
  thrackles}, Amer. Math. Monthly \textbf{118} (2011), no.~6, 544--548.

\bibitem{PP}
Amitai Perlstein and Rom Pinchasi, \emph{Generalized thrackles and geometric
  graphs in {$\Bbb R^3$} with no pair of strongly avoiding edges}, Graphs
  Combin. \textbf{24} (2008), no.~4, 373--389.

\bibitem{W}
D.~R. Woodall, \emph{Thrackles and deadlock}, Combinatorial {M}athematics and
  its {A}pplications ({P}roc. {C}onf., {O}xford, 1969), Academic Press, London,
  1971, pp.~335--347.

\end{thebibliography}

\end{document}